\documentclass[12pt, reqno]{amsart}
\usepackage{lipsum}
\usepackage{amsthm}

\makeatletter
\@namedef{subjclassname@2020}{%
\textup{2020} Mathematics Subject Classification}
\makeatother

\usepackage[margin=2.5cm]{geometry}
\usepackage{lineno,hyperref}
\usepackage{amsmath,amssymb,latexsym, amscd}
\usepackage{exscale, eps fig, graphics}
\usepackage{array}
\usepackage{xcolor}
\usepackage{float}
\restylefloat{table}
\newcommand\numberthis{\addtocounter{equation}{1}\tag{\theequation}}
\usepackage{algorithmic}
\usepackage{placeins}
\usepackage{longtable}
\usepackage{makecell}
\usepackage[utf8]{inputenc}
\usepackage[justification=centering]{caption}
\usepackage{graphicx}
\usepackage{caption}
\usepackage{subcaption}
\usepackage{mathtools}
\usepackage{eucal}
\usepackage{mathrsfs}
\hypersetup{pdfauthor={Name}}
\newtheorem{thm}{Theorem}[section]
\newtheorem{lem}[thm]{Lemma}
\newtheorem{prop}[thm]{Proposition}

\newcommand{\BigO}[1]{\ensuremath{\operatorname{O}\left(#1\right)}}

\renewcommand{\geq}{\geqslant}
\renewcommand{\leq}{\leqslant}

\newcommand{\BigOC}[1]{\ensuremath{\operatorname{O}_{\mathcal{C}_{\Lambda}}\left(#1\right)}}

\numberwithin{equation}{section}

\begin{document}

\baselineskip=17pt


\title{Pair correlation of Farey fractions with square-free denominators}

\author{Bittu Chahal}
\email{bittui@iiitd.ac.in}

\author{Sneha Chaubey}
\email{sneha@iiitd.ac.in}

\address[]{Department of Mathematics, IIIT Delhi, New Delhi 110020}

\begin{abstract}
In this article, we study the pair correlation of Farey fractions by proving that the limiting pair correlation function of the sequence of Farey fractions with square-free denominators exists and provide an explicit formula for the limiting pair correlation function.
\end{abstract}

 

\maketitle

\section{Introduction and main results}
The Farey sequence $\mathcal{F}_Q$ of order $Q$ is an ascending sequence of fractions $a/b$ in the unit interval $(0,1]$ such that $\gcd (a,b)=1$ and $0<a\leq b\leq Q.$ The Farey sequence plays a vital role in mathematics and is of independent interest for many mathematicians. It is well known that the Farey fractions in $\mathcal{F}_Q$  are nicely distributed in $[0,1]$ as $Q\to\infty.$ The primary interest in the distribution of the Farey fractions is due to the classical work of Franel \cite{Franel} and Landau \cite{Landau}, who showed that the Riemann hypothesis and quantitative statements about the uniform distribution of Farey fractions are known to be equivalent.

In particular, there is no general best way to measure the distribution of a sequence, but two ways are broadly accepted; one is to study the $h$-th level spacing measure, and the other is to study the $m$-level correlation measure. In the $h$-th level spacing measure, one aims to study how the spacings between the consecutive elements of a sequence are distributed. Hall \cite{Hall} studied the first level spacing distribution of Farey fractions by estimating the moments of spacings of consecutive Farey fractions. Augustin et al. \cite{Augustin} 
 studied the $h$-th level spacing distribution of Farey fractions for $h\geq 2$ by showing the convergence of the sequence of probability measures. In this note, we are interested in the $2$-level or pair correlation measure of Farey fractions. 

The study of correlations of sequences was introduced by physicists to perceive the spectra of high energies. A significant consideration has been given to these notions in several areas of number theory, mathematical physics, and probability theory. In number theory, it has received overwhelming attention after the work of Montgomery \cite{Montgomery} and Hejhal \cite{Hejhal} on the correlations of zeros of the Riemann zeta function, and Rudnick and Sarnak \cite{Rudnick} on the correlations of zeros of $L$-functions.

Let $\mathcal{F}$ be a finite set of $\mathscr{N}$ elements in the unit interval $[0,1]$. The pair correlation measure $\mathcal{S}_{\mathcal{F}}(I)$ of an interval $I\subset \mathbb{R}$ is defined as

\[\frac{1}{\mathscr{N}}\#\left\{(a,b)\in \mathcal{F}^2: a\ne b,\ a-b\in \frac{1}{\mathscr{N}}I+\mathbb{Z} \right\}.\]

The limiting pair correlation measure of an increasing sequence $(\mathcal{F}_n)_n$, for every interval $I$, is given (if it exists) by
\[\mathcal{S}(I)=\lim_{n\to \infty }\mathcal{S}_{\mathcal{F}_n}(I).\]
If
\[\mathcal{S}(I)=\int_{I}g(x)dx, \numberthis\label{g(x)}\]
then $g$ is called the limiting pair correlation function of $(\mathcal{F}_n)_n$. The pair correlation is said to be Poissonian if $g(x)\equiv 1$.

Boca and Zaharescu \cite{Boca} studied the pair correlation of Farey fractions and proved that the limiting pair correlation function of $\mathcal{F}_Q$ is given by
\[g(\lambda)=\frac{6}{\pi^2\lambda^2}\sum_{1\leq k<\frac{\pi^2\lambda}{3}}\phi(k)\log\frac{\pi^2\lambda}{3k},\numberthis\label{g(lambda)}\]
and it shows a strong repulsion between the elements of the sequence. The pair correlation of Farey fractions with prime denominators was studied by Xiong and Zaharescu \cite{Xiong}, who showed that the pair correlation is Poissonian. A more general result on the pair correlation of fractions with prime denominators is contained in \cite{Xiao}. Also, Xiong and Zaharescu \cite{Zaharescu}  studied the pair correlation of Farey fractions with denominators coprime with $B_Q$, the monotonic increasing sequence of square-free numbers with the condition that $B_{Q_1}|B_{Q_2}$ if $Q_1< Q_2$. They proved that the pair correlation of the sequence is Poissonian if $\lim_{Q\to\infty}\frac{\phi(B_Q)}{B_Q}=0$ and showed a strong repulsion if $\lim_{Q\to\infty}\frac{\phi(B_Q)}{B_Q}\ne 0$, where $\phi$ is the Euler's totient function. Recently, the pair correlation of Farey fractions with denominators coprime to $m$ was investigated by Boca and Siskaki \cite{Siskaki}, and the pair correlation function is given by
\[g_{(m)}(\lambda)=\frac{\phi(m)}{m}\cdot\frac{C_m}{\lambda^2}\sum_{1\leq \Delta\leq\frac{2\lambda}{C_m}}\phi(\Delta)\frac{(\Delta,m)}{\phi((\Delta,m))}\log\frac{2\lambda}{C_m\Delta}, \numberthis\label{g(m)}\]
where, $C_m=\frac{\phi(m)}{\zeta(2)m}\prod_{p|m}\left(1-\frac{1}{p^2} \right)^{-1}$. They also studied the pair correlation of Farey fractions with denominators in an arithmetic progression. The pair correlation measure of torsion points on elliptic curves is studied by Alkan et al. \cite{Alkan}, and they showed a strong repulsion between the torsion points of elliptic curves. The repulsion happens due to a correspondence between the distribution of Farey fractions and the distribution of the torsion points of elliptic points. Alkan et al. \cite{EmreA} also computed the pair correlation measure of the sum $\mathcal{F}_Q+\mathcal{F}_Q$ modulo $1$, as $Q\to\infty$. 

In this article, we are interested in the pair correlation of Farey fractions with square-free denominators. It is interesting to see the strong repulsion near zero, even after restricting the denominators to square-free. Authors in \cite{Ledoan} studied the distribution of the index of Farey fractions by imposing a similar restriction on Farey denominators in a fixed arithmetic progression. Note that the Farey fractions of order $Q$ with prime denominators lie in the set of Farey fractions of order $Q$ with square-free denominators but do not lie in the set of Farey fractions with denominators coprime to $B_{Q}\ne 1$, so our sequence of Farey fractions with square-free denominators does not coincide with the sequence in \cite{Zaharescu}. A positive integer $n$ is said to be square-free if there does not exist a prime $p$ such that $p^2|n$. 
Square-free numbers are closely related to prime numbers but have a positive asymptotic density with a more even distribution. This means one can expect a somewhat less random behavior with the pair correlation function being non-Poissonian, whereas, for prime denominators, it is Poissonian, as proved in \cite{Xiong}. We prove this in Theorem \ref{main result} by explicitly computing the limiting pair correlation function.

Denote
\[\mathcal{F}_{Q,2}:=\left\{\frac{a}{q}: 0<a\leq q\leq Q, (a,q)=1, q\ \text{is square-free} \right\}. \numberthis\label{F(Q)}\]
Throughout the paper, $p$ and $p^{\prime}$ denote prime numbers, $\tau(n)$ is the number of positive divisors of $n$, $\mu$ is the  M\"{o}bius function, and $(a,b)=1$ denotes that $a$ and $b$ are coprime. 
\begin{thm}\label{main result}
The limiting pair correlation function of the sequence $(\mathcal{F}_{Q,2})_Q$ exists and is given by
 \[\mathfrak{g}_2(\lambda)=\frac{6}{\lambda^2\pi^2}\sum_{1\leq m<\frac{\lambda\pi^2}{3\delta(p)}}F(m)\log\frac{\lambda\pi^2}{3m\delta(p)}\ \numberthis\label{g}\]
for any $\lambda\geq 0,$ where
 $\delta(p)=\prod_{p}\left(1-\frac{1}{p(p+1)}\right),$
 and
 \begin{align*}
  F(m)=&m\sum_{d|m}\frac{\mu(d)\phi(d)}{d^2}\prod_{\substack{p\\(p,d)=1}}\left(1-\frac{1}{p(p+1)}\right)\left(1-\frac{1}{p^2+p-1}\right)\\
      &\times\prod_{\substack{p|m\\ (p,d)=1}}\left(1-\frac{(p-1)(p^2+p-2)}{p^3(p+1)}\prod_{\substack{p^{\prime}|m\\(p^{\prime},dp)=1}}\left(1+\frac{p^{\prime}-1}{p{^{\prime}}^2+p^{\prime}-2}\right)\right)\\
     &\times\prod_{\substack{p|m\\ p|d}}\left(1-\frac{1}{p+1}\prod_{\substack{p^{\prime}|m\\(p^{\prime},dp)=1}}\left(1+\frac{p^{\prime}-1}{p{^{\prime}}^2+p^{\prime}-2}\right)\right).\numberthis\label{F(m)}
      \end{align*}
\end{thm}

Note that $F(m)$ is finite for every $m$ since each product term is bounded, and the sum runs over the positive divisors of $m$. Since there are finite terms in the sum in \eqref{g} as $1\leq m<\frac{\lambda\pi^2}{3\delta(p)}$, the function $\mathfrak{g}_2(\lambda)$ is well defined with support $\left(\frac{3\delta(p)}{\pi^2},\infty\right)$.

\begin{figure}[h]
\centering
\subfloat{
\includegraphics[width=8cm, height=5cm]{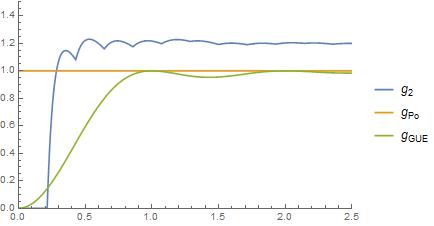}}
\subfloat{\includegraphics[width=8cm, height=5cm]{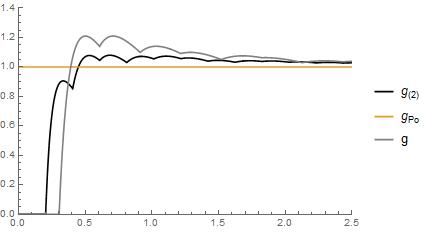}}
\caption{The graphs of $\mathfrak{g}_2(\lambda), g_{Po}(\lambda), g_{GUE}(\lambda), g_{(2)}(\lambda)$ and $g(\lambda)$ }
\end{figure} 

Figure $1$ shows the graph of $\mathfrak{g}_2(\lambda)$ and to compare, we plot the graphs of the pair correlation functions of GUE model, Poisson case, Farey fractions, and Farey fractions with coprimality condition, which are $g_{GUE}(\lambda)=1-\left(\frac{\sin \pi\lambda}{\pi\lambda} \right)^2$, $g_{Po}(\lambda)=1$, $g(\lambda)$ in \eqref{g(lambda)} and $g_{(2)}(\lambda)$ in \eqref{g(m)}  respectively. The graph of $\mathfrak{g}_2(\lambda)$ shows a strong repulsion between the elements of the sequence $\mathcal{F}_{Q,2}$, even more robust than the 
 repulsion amongst the zeros of the Riemann zeta function. As $\lambda\to\infty$, repulsion decreases, and distribution tends to become constant.


\section{Preliminaries}
In this section, we derive results which will be used in proving Theorem\ref{main result}. We begin with counting the number of Farey fractions of order $Q$ with square-free denominators. 
\begin{prop}
Let $\mathcal{F}_{Q,2}$ as in \eqref{F(Q)}, then
\[\mathscr{N}_{Q}=\#\mathcal{F}_{Q,2}=\frac{3Q^2}{\pi^2}\prod_p\left(1-\frac{1}{p(p+1)}\right)+\BigO{Q^{\frac{3}{2}}}.\numberthis\label{N(Q)}\]
\end{prop}
\begin{proof}
  We have 
  \begin{align*}
\#\mathcal{F}_{Q,2}&=\sum_{\substack{s\leq Q\\ s\ \text{is square-free} }}\phi(s)=\sum_{s\leq Q}\phi(s)\mu(s)^2
      =\sum_{d\leq Q}\frac{\mu(d)}{d}\sum_{\substack{s\leq Q\\ d|s}}\mu(s)^2 s\\ &=\sum_{d\leq Q}\mu(d)\sum_{j\leq \frac{Q}{d}}\sum_{\substack{j\leq s_1\leq \frac{Q}{d}\\ (s_1,d)=1}}\mu(s_1)^2. \numberthis\label{numberF}
  \end{align*}
  The inner sum in \ref{numberF} is computed using \cite[Theorem 1]{Suryanarayana}
  \[\sum_{\substack{ s\leq \frac{Q}{d}\\ (s,d)=1}}\mu(s)^2=\frac{1}{\zeta(2)}\cdot Q\cdot\frac{\phi(d)}{d^2}\prod_{p|d}\left(1-\frac{1}{p^2}\right)^{-1}+\BigO{\left(\frac{Q}{d}\right)^{1/2}}.\]
  Inserting this into \eqref{numberF}, we obtain
  \begin{align*}
      \#\mathcal{F}_{Q,2}=&\sum_{d\leq Q}\mu(d)\sum_{j\leq \frac{Q}{d}}\left(\frac{Q\phi(d)}{d^2\zeta(2)}\prod_{p|d}\left(1-\frac{1}{p^2}\right)^{-1}+\BigO{\left(\frac{Q}{d}\right)^{1/2}}\right)\\ &-\sum_{d\leq Q}\mu(d)\sum_{j\leq \frac{Q}{d}}\left(\frac{j\phi(d)}{d\zeta(2)}\prod_{p|d}\left(1-\frac{1}{p^2}\right)^{-1}+\BigO{j^{1/2}}\right) \\
      =&\frac{Q^2}{2\zeta(2)}\sum_{d\leq Q}\frac{\mu(d)\phi(d)}{d^3}\prod_{p|d}\left(1-\frac{1}{p^2}\right)^{-1}+\BigO{Q^\frac{3}{2}}\\
      =&\frac{3Q^2}{\pi^2}\sum_{d=1}^{\infty}\frac{\mu(d)\phi(d)}{d^3}\prod_{p|d}\left(1-\frac{1}{p^2}\right)^{-1}+\BigO{Q^\frac{3}{2}}\\
       =&\frac{3Q^2}{\pi^2}\prod_{p}\left(1-\frac{1}{p(p+1)}\right)+\BigO{Q^\frac{3}{2}}.
  \end{align*}
\end{proof}
We next prove a formula for exponential sums over Farey fractions in $\mathcal{F}_{Q,2}$.
\begin{prop}\label{prop2}
     Let $r\in \mathbb{Z},$ we have
        \[\sum_{\gamma\in \mathcal{F}_{Q,2}}e(r\gamma)=\sum_{d\leq Q}\mu(d)\sum_{\substack{q\leq\frac{Q}{d}\\(q,d)=1,\ q|r}}q\mu(q)^2,\]
        where $e(x)=\exp{(2\pi ix)}.$
\end{prop}
\begin{proof}
We have
\begin{align*}
    \sum_{\gamma\in \mathcal{F}_{Q,2}}e(r\gamma)&=\sum_{\substack{q\leq Q\\\mu(q)^2=1}}\sum_{\substack{1\leq a\leq q\\(a,q)=1}}e\left(\frac{ar}{q} \right)
    =\sum_{\substack{q\leq Q\\\mu(q)^2=1}}\sum_{1\leq a\leq q}e\left(\frac{ar}{q} \right)\sum_{d|\gcd(a,q)}\mu(d)\\
    &=\sum_{d\leq Q}\mu(d) \sum_{\substack{q\leq Q\\\mu(q)^2=1, d|q}}\sum_{\substack{1\leq a\leq q\\ d|a}}e\left(\frac{ar}{q} \right)\\
    &=\sum_{d\leq Q}\mu(d) \sum_{\substack{q_1\leq \frac{Q}{d}\\ \mu(q_1)^2=1, (q_1,d)=1}}\sum_{1\leq a_1\leq q_1}e\left(\frac{a_1r}{q_1} \right)\\
    &=\sum_{d\leq Q}\mu(d)\sum_{\substack{q_1\leq \frac{Q}{d}\\ (q_1,d)=1, q_1|r}}q_1\mu(q_1)^2.
\end{align*}
    \end{proof}
  We next derive an integral formula for lattice points in a bounded domain satisfying certain divisibility constraints. 
  \begin{lem}\label{key lemma}
     Let $\Omega\subset [1, R]^2$ be a bounded region, and $f$ is a continuously differentiable function on $\Omega.$ For any positive integers $r_1$ and $r_2$, we have
     \[\sum_{\substack{(a,b)\in \Omega\cap\mathbb{Z}^2\\(a,r_1)=(b,r_2)=(a,b)=1\\\mu(a)^2=\mu(b)^2=1}}f(a,b)=\frac{6P}{\pi^2}\iint_{\Omega}f(x,y)dxdy+E,\numberthis\label{2.3}\] 
     where
      \begin{align*} 
  P=\frac{\phi(r_1)\phi(r_2)}{r_1r_2}\prod_{p|\gcd(r_1, r_2)}\left(1-\frac{1}{p^2}\right)^{-1}\prod_{\substack{p\\ (p,r_1r_2)=1}}\left(1-\frac{1}{p(p+1)}\right)\left(1-\frac{1}{p^2+p-1}\right),
     \end{align*}
     and
       \[ E\ll \left(\tau(r_1)\left\|{\frac{\partial f}{\partial x}}\right\|_{\infty}+\tau(r_2)\left\|{\frac{\partial f}{\partial y}}\right\|_{\infty}\right)\text{Area}(\Omega)\sqrt{R}\log^2 R+\|f\|_{\infty}(\tau(r_1)+\tau(r_2))R^{\frac{3}{2}}\log^2 R.\]
    \end{lem}
    \begin{proof}
      We begin with removing the square-free conditions on the left side of \eqref{2.3} by M\"{o}bius summation and obtain
        \begin{align*}
           S_{f,\Omega,r_1,r_2}:&=\sum_{\substack{(a,b)\in \Omega\cap\mathbb{Z}^2\\(a,r_1)=(b,r_2)=(a,b)=1\\\mu(a)^2=\mu(b)^2=1}}f(a,b)=\sum_{\substack{(a,b)\in \Omega\cap\mathbb{Z}^2\\(a,r_1)=(b,r_2)=(a,b)=1}}f(a,b)\sum_{d_1^2|a}\mu(d_1)\sum_{d_2^2|b}\mu(d_2)\\
           &=\sum_{d_1^2,d_2^2\leq R}\mu(d_1)\mu(d_2)\sum_{\substack{(a,b)\in \Omega\cap\mathbb{Z}^2\\(a,r_1)=(b,r_2)=(a,b)=1\\d_1^2|a,\ d_2^2|b}}f(a,b)\\
            &=\sum_{d_1^2,d_2^2\leq R}\mu(d_1)\mu(d_2)\sum_{\substack{(a_0,b_0)\in \Omega_{(d_1,d_2)}\cap\mathbb{Z}^2\\(d_1^2a_0,r_1)=(d_2^2b_0,r_2)=1\\(d_1^2a_0,d_2^2b_0)=1}}f_{(d_1,d_2)}(a_0,b_0)\\
            &=\sum_{\substack{d_1^2,d_2^2\leq R\\(d_1,r_1)=(d_2,r_2)=1\\(d_1,d_2)=1}}\mu(d_1)\mu(d_2)\sum_{\substack{(a_0,b_0)\in \Omega_{(d_1,d_2)}\cap\mathbb{Z}^2\\(a_0,r_1d_2)=(b_0,r_2d_1)=1\\(a_0,b_0)=1}}f_{(d_1,d_2)}(a_0,b_0),\numberthis\label{sum7}
       \end{align*}
       where $f_{(d_1,d_2)}(a_0,b_0)=f(d_1^2a_0,d_2^2b_0)$ and $\Omega_{(d_1,d_2)}=\left\{(x,y)|\ x\in\frac{1}{d_1^2}[1,R], y\in\frac{1}{d_2^2}[1,R]\right\}.$
 The inner sum in \eqref{sum7} can be estimated using  M\"{o}bius summation thereby removing the coprimality condition 
 
     \begin{align*}
         S_{f,\Omega,r_1,r_2}^{1}:&=\sum_{\substack{(a_0,b_0)\in \Omega_{(d_1,d_2)}\cap\mathbb{Z}^2\\(a_0,r_1d_2)=(b_0,r_2d_1)=1\\(a_0,b_0)=1}}f_{(d_1, d_2)}(a_0, b_0)
        =\sum_{\substack{(a_0,b_0)\in \Omega_{(d_1,d_2)}\cap\mathbb{Z}^2\\(a_0,r_1d_2)=(b_0,r_2d_1)=1}}f_{(d_1,d_2)}(a_0,b_0)\sum_{d|\gcd(a_0,b_0)}\mu(d)\\
        &=\sum_{d\leq\frac{R}{\max(d_1^2,d_2^2)}}\mu(d)\sum_{\substack{(a_1,b_1)\in \frac{1}{d}\Omega_{(d_1,d_2)}\cap\mathbb{Z}^2\\(da_1,r_1d_2)=(db_1,r_2d_1)=1}}f_{(d_1,d_2)}(da_1, db_1)\\
         &=\sum_{\substack{d\leq\frac{R}{\max(d_1^2,d_2^2)}\\(d,r_1r_2d_1d_2)=1}}\mu(d)\sum_{\substack{(a_1,b_1)\in \frac{1}{d}\Omega_{(d_1,d_2)}\cap\mathbb{Z}^2\\(a_1,r_1d_2)=(b_1,r_2d_1)=1}}f_{(d_1,d_2)}(da_1, db_1)\\
          &=\sum_{\substack{d\leq\frac{R}{\max(d_1^2,d_2^2)}\\(d,r_1r_2d_1d_2)=1}}\mu(d)\sum_{(a_1,b_1)\in \frac{1}{d}\Omega_{(d_1,d_2)}\cap\mathbb{Z}^2}f_{(d_1,d_2)}(da_1, db_1)\sum_{s|\gcd(a_1,r_1d_2)}\mu(s)\sum_{t|\gcd(b_1,r_2d_1)}\mu(t)\\
           &=\sum_{\substack{d\leq\frac{R}{\max(d_1^2,d_2^2)}\\(d,r_1r_2d_1d_2)=1}}\mu(d)\sum_{s|r_1d_2,\ t|r_2d_1}\mu(s)\mu(t)\sum_{(a^{\prime},b^{\prime})\in \Delta\cap\mathbb{Z}^2}\mathscr{G}(a^{\prime},b^{\prime}),\numberthis\label{sum8}
    \end{align*}
    where $\mathscr{G}(a^{\prime},b^{\prime})=f(dsd_1^2a^{\prime},dtd_2^2b^{\prime})$ and $\Delta=\left\{(x,y)|\ x\in\frac{1}{dsd_1^2}[1,R], y\in\frac{1}{dtd_2^2}[1,R]\right\}$.  We use \cite[Lemma 1]{Cobeli} to estimate the innermost sum in \eqref{sum8} and obtain
    \small\begin{align*}
        \sum_{(a^{\prime},b^{\prime})\in \Delta\cap\mathbb{Z}^2}\mathscr{G}(a^{\prime},b^{\prime}) =&\iint_{\Delta}\mathscr{G}(x,y)dxdy+\BigO{\text{Area}(\Delta)\left(\left\|{\frac{\partial \mathscr{G}}{\partial x}}\right\|_{\infty}+\left\|{\frac{\partial \mathscr{G}}{\partial y}}\right\|_{\infty}\right)}\\&+\BigO{\|\mathscr{G}\|_{\infty}(1+\text{length}(\partial\Delta))}\\
        =&\frac{1}{std^2d_1^2d_2^2}\iint_{\Omega}f(x,y)dxdy +\BigO{\left(\frac{1}{dtd_2^2}\left\|{\frac{\partial f}{\partial x}}\right\|_{\infty}+\frac{1}{dsd_1^2}\left\|{\frac{\partial f}{\partial y}}\right\|_{\infty}\right)\text{Area}(\Omega)}\\ &+\BigO{\|f\|_{\infty}\frac{R}{d}\left(\frac{1}{sd_1^2}+\frac{1}{td_2^2}\right)}.\numberthis\label{sum9}
           \end{align*}
Inserting \eqref{sum9} in \eqref{sum8}, we get
\begin{align*}
     S_{f,\Omega,r_1,r_2}^{1}=&\frac{1}{d_1^2d_2^2}\sum_{\substack{d\leq\frac{R}{\max(d_1^2,d_2^2)}\\(d,r_1r_2d_1d_2)=1}}\frac{\mu(d)}{d^2}\sum_{s|r_1d_2,\ t|r_2d_1}\frac{\mu(s)\mu(t)}{st}\iint_{\Omega}f(x,y)dxdy\\ &+\BigO{\text{Area}(\Omega)\log^2 R\left(\frac{\tau(r_1d_2)}{d_2^2}\left\|{\frac{\partial f}{\partial x}}\right\|_{\infty}+\frac{\tau(r_2d_1)}{d_1^2}\left\|{\frac{\partial f}{\partial y}}\right\|_{\infty}\right)}\\ &+\BigO{R\log^2R\ \|f\|_{\infty}\left(\frac{\tau(r_2d_1)}{d_1^2}+\frac{\tau(r_1d_2)}{d_2^2}\right)}.\numberthis\label{sum10}
    \end{align*}
  The summation in \eqref{sum10} is estimated as
  \begin{align*}
       S_{r_1,r_2}:&=\sum_{\substack{d\leq\frac{R}{\max(d_1^2,d_2^2)}\\(d,r_1r_2d_1d_2)=1}}\frac{\mu(d)}{d^2}\sum_{s|r_1d_2,\ t|r_2d_1}\frac{\mu(s)\mu(t)}{st}\\
      &=\left(\sum_{\substack{d=1\\(d,r_1r_2d_1d_2)=1}}^{\infty}\frac{\mu(d)}{d^2}+\BigO{\frac{\max(d_1^2,d_2^2)}{R}}\right)\sum_{s|r_1d_2}\frac{\mu(s)}{s}\sum_{t|r_2d_1}\frac{\mu(t)}{t}\\
      &=\prod_{\substack{p\\(p,r_1r_2d_1d_2)=1}}\left(1-\frac{1}{p^2}\right)\prod_{p|r_1d_2}\left(1-\frac{1}{p}\right)\prod_{p|r_2d_1}\left(1-\frac{1}{p}\right)+\BigO{\frac{\max(d_1^2,d_2^2)}{R}}\\
      &=\frac{1}{\zeta(2)}\prod_{p|r_1r_2d_1d_2}\left(1-\frac{1}{p^2}\right)^{-1}\prod_{p|r_1d_2}\left(1-\frac{1}{p}\right)\prod_{p|r_2d_1}\left(1-\frac{1}{p}\right)+\BigO{\frac{\max(d_1^2,d_2^2)}{R}}.\numberthis\label{sum11}
  \end{align*}
So, \eqref{sum11} in conjunction with \eqref{sum10} and \eqref{sum7}, gives
\small\begin{align*}
    S_{f,\Omega,r_1,r_2}=&\frac{1}{\zeta(2)}\sum_{\substack{d_1^2,d_2^2\leq R\\(d_1,r_1)=(d_2,r_2)=1\\(d_1,d_2)=1}}\frac{\mu(d_1)\mu(d_2)}{d_1^2d_2^2}\\&\times\prod_{p|r_1r_2d_1d_2}\left(1-\frac{1}{p^2}\right)^{-1}\prod_{p|r_1d_2}\left(1-\frac{1}{p}\right)\prod_{p|r_2d_1}\left(1-\frac{1}{p}\right)\iint_{\Omega}f(x,y)dxdy+ E\\
    =&\frac{1}{\zeta(2)}\prod_{p|r_1r_2}\left(1-\frac{1}{p^2}\right)^{-1}\prod_{p|r_1}\left(1-\frac{1}{p}\right)\prod_{p|r_2}\left(1-\frac{1}{p}\right)\sum_{\substack{d_1^2,d_2^2\leq R\\(d_1,r_1)=(d_2,r_2)=1\\(d_1,d_2)=1}}\frac{\mu(d_1)\mu(d_2)}{d_1^2d_2^2}\\&\times\prod_{\substack{p|d_1d_2\\(p,r_1r_2)=1}}\left(1-\frac{1}{p^2}\right)^{-1}\prod_{\substack{p|d_1\\(p,r_2)=1}}\left(1-\frac{1}{p}\right)\prod_{\substack{p|d_2\\(p,r_1)=1}}\left(1-\frac{1}{p}\right)\iint_{\Omega}f(x,y)dxdy+ E.\numberthis\label{sum12}
\end{align*}
To estimate the summation in \eqref{sum12}, we write
\small\begin{align*}
    S_{r_1,r_2}^{11}:=&\sum_{\substack{d_1^2,d_2^2\leq R\\(d_1,r_1)=(d_2,r_2)=1\\(d_1,d_2)=1}}\frac{\mu(d_1)\mu(d_2)}{d_1^2d_2^2}\prod_{\substack{p|d_1d_2\\(p,r_1r_2)=1}}\left(1-\frac{1}{p^2}\right)^{-1}\prod_{\substack{p|d_1\\(p,r_2)=1}}\left(1-\frac{1}{p}\right)\prod_{\substack{p|d_2\\(p,r_1)=1}}\left(1-\frac{1}{p}\right)\\
    =&\sum_{\substack{d_1^2\leq R\\(d_1,r_1)=1}}\frac{\mu(d_1)}{d_1^2}\prod_{\substack{p|d_1\\(p,r_1r_2)=1}}\left(1-\frac{1}{p^2}\right)^{-1}\prod_{\substack{p|d_1\\(p,r_2)=1}}\left(1-\frac{1}{p}\right)\\&\times\sum_{\substack{d_2^2\leq R\\(d_2,r_2)=(d_1,d_2)=1}}\frac{\mu(d_2)}{d_2^2}\prod_{\substack{p|d_2\\(p,r_1r_2d_1)=1}}\left(1-\frac{1}{p^2}\right)^{-1}\prod_{\substack{p|d_2\\(p,r_1)=1}}\left(1-\frac{1}{p}\right).\numberthis\label{sum13}
\end{align*}
The inner sum in \eqref{sum13} is estimated as
\small\begin{align*}
    S_{r_1,r_2}^{{11}^{\prime}}:=&\sum_{\substack{d_2^2\leq R\\(d_2,d_1r_2)=1}}\frac{\mu(d_2)}{d_2^2}\prod_{\substack{p|d_2\\(p,r_1r_2d_1)=1}}\left(1-\frac{1}{p^2}\right)^{-1}\prod_{\substack{p|d_2\\(p,r_1)=1}}\left(1-\frac{1}{p}\right)\\
     =&\prod_{\substack{p|r_1\\(p,r_2)=1}}\left(1-\frac{1}{p^2}\right)\prod_{\substack{p|d_1\\(p,r_2)=1, (p,r_1)\ne 1}}\left(1-\frac{1}{p^2}\right)^{-1}\prod_{\substack{p\\(p,r_1r_2)=1}}\left(1-\frac{1}{p(p+1)}\right)\\&\times\prod_{\substack{p|d_1\\(p,r_1r_2)=1}}\left(1-\frac{1}{p(p+1)}\right)^{-1}+\BigO{\frac{1}{\sqrt{R}}}.\numberthis\label{sum14}
\end{align*}
So, inserting \eqref{sum14} into \eqref{sum13}, we obtain
\small\begin{align*}
    S_{r_1,r_2}^{11}=&\prod_{\substack{p|r_1\\(p,r_2)=1}}\left(1-\frac{1}{p^2}\right)\prod_{\substack{p\\(p,r_1r_2)=1}}\left(1-\frac{1}{p(p+1)}\right)\sum_{\substack{d_1^2\leq R\\(d_1,r_1)=1}}\frac{\mu(d_1)}{d_1^2}\prod_{\substack{p|d_1\\(p,r_1r_2)=1}}\left(1-\frac{1}{p^2}\right)^{-1}\\&\times\prod_{\substack{p|d_1\\(p,r_2)=1}}\left(1-\frac{1}{p}\right)\prod_{\substack{p|d_1\\(p,r_2)=1, (p,r_1)\ne 1}}\left(1-\frac{1}{p^2}\right)\prod_{\substack{p|d_1\\(p,r_1r_2)=1}}\left(1-\frac{1}{p(p+1)}\right)^{-1}+\BigO{\frac{1}{\sqrt{R}}}\\
    =&\prod_{\substack{p|r_1\\(p,r_2)=1}}\left(1-\frac{1}{p^2}\right)\prod_{\substack{p|r_2\\(p,r_1)=1}}\left(1-\frac{1}{p^2}\right)\prod_{\substack{p\\(p,r_1r_2)=1}}\left(1-\frac{1}{p(p+1)}\right)\left(1-\frac{1}{p^2+p-1}\right)+\BigO{\frac{1}{\sqrt{R}}}.\numberthis\label{sum15}
\end{align*}
Inserting \eqref{sum15} in \eqref{sum12} gives the required result.

\end{proof}
We also use the Poisson summation formula which we state below.
  \begin{prop}\cite[p. 538]{Vaughan}\label{Poisson} (Poisson's summation formula).
    Let $f\in L^1(\mathbb{R})$ and $\widehat{f}$ be the Fourier transform of $f$, then we have
    \[\sum_{n=-\infty}^{\infty}f(n)=\sum_{m=-\infty}^{\infty}\widehat{f}(m). \]
  \end{prop} 
\section{Proof of Theorem \ref{main result}}
 We aim to estimate, for any positive real number $\Lambda
 $, the quantity
\[S(\Lambda)=\frac{1}{\mathscr{N}_Q}\#\left\{(\gamma_1,\gamma_2)\in \mathcal{F}_{Q,2}^2: \gamma_1\ne\gamma_2, \gamma_1-\gamma_2\in\frac{1}{\mathscr{N}_Q}(0,\Lambda)+\mathbb{Z}\right\},\]
as $Q\to \infty.$ Let $H$ be any continuously differentiable function with Supp\ $H\subset(0,\Lambda).$  Define
\[h(y)=\sum_{n\in\mathbb{Z}}H(\mathscr{N}_Q(y+n)),\ y\in\mathbb{R},\]  and
\[S=\sum_{\substack{\gamma_1,\gamma_2\in \mathcal{F}_{Q,2}\\\gamma_1\ne\gamma_2 }}h(\gamma_1-\gamma_2) 
 =\sum_{\substack{\gamma_1,\gamma_2\in \mathcal{F}_{Q,2}\\n\in\mathbb{Z}}}H(\mathscr{N}_{Q}(\gamma_1-\gamma_2+n)), \numberthis\label{sum}\]
 since Supp$H\subset(0,\Lambda)$, the condition $\gamma_1$ and $\gamma_2$ are distinct can be removed for $Q$ large such that $\mathscr{N}_{Q}>\Lambda.$ Let
 \[h(y)=\sum_{n\in \mathbb{Z}}c_ne(ny)\]
 be the Fourier series expansion of $h$, with the Fourier coefficient 
 \begin{align*}
     c_n&=\int_0^1h(x)e(-nx)dx=\sum_{m\in\mathbb{Z}}\int_0^1H(\mathscr{N}_{Q}(x+m))e(-nx)dx\\
     &=\int_{\mathbb{R}}H(\mathscr{N}_{Q}v)e(-nv)dv=\frac{1}{\mathscr{N}_{Q}}\widehat{H}\left(\frac{n}{\mathscr{N}_{Q}}\right),
 \end{align*}
 where $\widehat{H}$ is the Fourier transform of $H.$ Then by \eqref{sum}, we have
   \begin{align*}
     S&=\sum_{\gamma_1,\gamma_2\in \mathcal{F}_{Q,2}}h(\gamma_1-\gamma_2)
     =\sum_{\substack{\gamma_1,\gamma_2\in \mathcal{F}_{Q,2}\\n\in\mathbb{Z}}}c_ne(n(\gamma_1-\gamma_2))\\
     &=\sum_{n\in \mathbb{Z}}c_n\sum_{\gamma_1\in \mathcal{F}_{Q,2}}e(n\gamma_1)\sum_{\gamma_2\in \mathcal{F}_{Q,2}}e(n\gamma_2).\numberthis\label{exp}
 \end{align*}
Using Proposition \ref{prop2} in \eqref{exp}, we obtain
 \begin{align*}
     S&=\sum_{n\in\mathbb{Z}}c_n\sum_{d_1\leq Q}\mu(d_1)\sum_{\substack{q_1\leq \frac{Q}{d_1}\\(q_1,d_2)=1, q_1|n}}q_1\mu(q_1)^2 \sum_{d_2\leq Q}\mu(d_2)\sum_{\substack{q_2\leq \frac{Q}{d_2}\\(q_2,d_2)=1, q_2|n}}q_2\mu(q_2)^2\\
     &=\sum_{d_1, d_2\leq Q}\mu(d_1)\mu(d_2)\sum_{\substack{q_1\leq \frac{Q}{d_1}, q_2\leq \frac{Q}{d_2}\\(q_1,d_1)=1, (q_2,d_2)=1}}q_1q_2\mu(q_1)^2\mu(q_2)^2\sum_{r\in\mathbb{Z}}c_{[q_1,q_2]r}.\numberthis\label{sum2}
 \end{align*}
To estimate the innermost sum, we consider the function: For each $y>0$
\[H_y(x)=\frac{1}{y}H\left(\frac{\mathscr{N}_{Q}x}{y}\right),\  x\in \mathbb{R}.\]
Then
\[\widehat{H}_y(z)=\frac{1}{\mathscr{N}_Q}\widehat{H}\left(\frac{yz}{\mathscr{N}_Q} \right).\]
Using the Fourier transform and a suitable change of variable, we have
\begin{align*}
    c_{[q_1,q_2]r}&=\int_{\mathbb{R}}H(\mathscr{N}_{Q}t)e(-[q_1,q_2]rt)dt\\
    &=\int_{\mathbb{R}}\frac{1}{[q_1,q_2]}H\left(\frac{\mathscr{N}_{Q}u}{[q_1,q_2]} \right)e(-ru)du\\
    &=\int_{\mathbb{R}}H_{[q_1,q_2]}(u)e(-ru)du=\widehat{H}_{[q_1,q_2]}(r).
\end{align*}
Using Proposition \ref{Poisson}, we obtain
\begin{align*}
    \sum_{r\in\mathbb{Z}}c_{[q_1,q_2]r}&=\sum_{r\in\mathbb{Z}}\widehat{H}_{[q_1,q_2]}(r)
    =\sum_{r\in\mathbb{Z}}H_{[q_1,q_2]}(r)
    =\sum_{r\in\mathbb{Z}}\frac{1}{[q_1,q_2]}H\left(\frac{\mathscr{N}_{Q}r}{[q_1,q_2]} \right).\numberthis\label{Fourier coefficient}
\end{align*}
Combining \eqref{sum2} and \eqref{Fourier coefficient}, we get
\begin{align*}
    S &=\sum_{d_1, d_2\leq Q}\mu(d_1)\mu(d_2)\sum_{\substack{q_1\leq \frac{Q}{d_1}, q_2\leq \frac{Q}{d_2}\\(q_1,d_1)=1, (q_2,d_2)=1}}q_1q_2\mu(q_1)^2\mu(q_2)^2\sum_{r\in\mathbb{Z}}\frac{1}{[q_1,q_2]}H\left(\frac{\mathscr{N}_{Q}r}{[q_1,q_2]} \right)\\
    &=\sum_{d_1, d_2\leq Q}\mu(d_1)\mu(d_2)\sum_{\substack{q_1\leq \frac{Q}{d_1}, q_2\leq \frac{Q}{d_2}\\(q_1,d_1)=1, (q_2,d_2)=1}}\gcd(q_1,q_2)\mu(q_1)^2\mu(q_2)^2\sum_{r\in\mathbb{Z}}H\left(\frac{\mathscr{N}_{Q}r}{[q_1,q_2]} \right).\numberthis\label{sum3}
\end{align*}
Let $\gcd(q_1,q_2)=\beta,$ so that $q_1=q_1^{\prime}\beta,\ q_2= q_2^{\prime}\beta$ with $(q_1^{\prime},q_2^{\prime})=1$. Then \eqref{sum3} becomes

\begin{align*}
    S &=\sum_{d_1, d_2\leq Q}\mu(d_1)\mu(d_2)\sum_{\beta\leq\frac{Q}{\max\{d_1,d_2\}}}\beta \sum_{\substack{q_1^{\prime}\leq \frac{Q}{\beta d_1}, q_2^{\prime}\leq \frac{Q}{\beta d_2}\\(q_1^{\prime}\beta,d_1)=1, (q_2^{\prime}\beta,d_2)=1\\ (q_1^{\prime},q_2^{\prime})=1}}\mu(q_1^{\prime}\beta)^2\mu(q_2^{\prime}\beta)^2\sum_{r\in\mathbb{Z}}H\left(\frac{\mathscr{N}_{Q}r}{q_1^{\prime}q_2^{\prime}\beta} \right)\\
    &=\sum_{d_1, d_2\leq Q}\mu(d_1)\mu(d_2)\sum_{\substack{\beta\leq\frac{Q}{\max\{d_1,d_2\}}\\(\beta,d_1d_2)=1}}\beta\mu(\beta)^2 \sum_{\substack{q_1^{\prime}\leq \frac{Q}{\beta d_1}, q_2^{\prime}\leq \frac{Q}{\beta d_2}\\(q_1^{\prime},\beta d_1)=1, (q_2^{\prime},\beta d_2)=1\\(q_1^{\prime}, q_2^{\prime})=1}}\mu(q_1^{\prime})^2\mu(q_2^{\prime})^2\sum_{r\in\mathbb{Z}}H\left(\frac{\mathscr{N}_{Q}r}{q_1^{\prime}q_2^{\prime}\beta} \right).\numberthis\label{sum4}
 \end{align*}
For non-zero contribution from $H$, using the fact that Supp$H \subset (0,\Lambda)$ and \eqref{N(Q)}, one must have
\[0<\frac{\mathscr{N}_{Q}r}{q_1^{\prime}q_2^{\prime}\beta}<\Lambda,\numberthis\label{nonzero contribution}\] which implies 
\[\beta d_1d_2r<\frac{\Lambda\pi^2}{3\delta(p)}=\mathcal{C}_{\Lambda}.\]
By applying above estimate and observing that
\[H\left(\frac{\mathscr{N}_{Q}r}{q_1^{\prime}q_2^{\prime}\beta} \right)=H\left(\frac{3Q^2\delta(p)r}{\pi^2q_1^{\prime}q_2^{\prime}\beta} \right)+\BigO{\frac{r}{q_1^{\prime}q_2^{\prime}\beta}Q^{\frac{3}{2}}},\]
the sum in \eqref{sum4} can be expressed as
\small\[S=\sum_{\substack{d_1, d_2,\beta, r\geq 1\\ \beta d_1d_2r<\mathcal{C}_{\Lambda}\\(\beta, d_1d_2)=1}}\beta\mu(d_1)\mu(d_2)\mu(\beta)^2 \sum_{\substack{q_1^{\prime}\leq \frac{Q}{\beta d_1}, q_2^{\prime}\leq \frac{Q}{\beta d_2}\\(q_1^{\prime},\beta d_1)=1, (q_2^{\prime},\beta d_2)=1\\(q_1^{\prime}, q_2^{\prime})=1}}\mu(q_1^{\prime})^2\mu(q_2^{\prime})^2H\left(\frac{3Q^2\delta(p)r}{\pi^2q_1^{\prime}q_2^{\prime}\beta} \right)+\BigOC{Q^{\frac{3}{2}}\log^2 Q}.\numberthis\label{sum5}\]

To estimate the inner sum in \eqref{sum5}, we use Lemma \ref{key lemma}, which counts the lattice points with some coprimality conditions and square-free restrictions in a bounded region.
Note that, since Supp $H\subset(0,\Lambda)$, then for non-zero contribution from $H$, one has $0<\frac{3Q^2\delta(p)r}{\pi^2x_1x_2\beta}<\Lambda$. For $0<x_1\leq \frac{Q}{\beta d_1}$ and $0<x_2\leq \frac{Q}{\beta d_2}$, we obtain
\[\frac{1}{x_1}\leq \frac{\mathcal{C}_{\Lambda}}{rd_2Q}\ \text{and}\ \frac{1}{x_2}\leq \frac{\mathcal{C}_{\Lambda}}{rd_1Q}. \numberthis\label{x_1}\]
Using \eqref{x_1} and the necessary condition for the non-zero contribution of $H$,  we get
\[\left|\frac{\partial H}{\partial x_1}(x_1,x_2)\right|\ll\frac{1}{Q}\ \text{and}\ \left|\frac{\partial H}{\partial x_2}(x_1,x_2)\right|\ll\frac{1}{Q}.\]
Hence
\[\|DH\|_{\infty}\ll\frac{1}{Q}.\]

We apply Lemma \ref{key lemma} with $r_1=\beta d_1,\ r_2=\beta d_2$, and $f(a,b)=H\left(\frac{3r\delta(p)Q^2}{\beta\pi^2ab} \right)$, to obtain
\begin{align*}
    \sum_{\substack{q_1^{\prime}\leq \frac{Q}{\beta d_1}, q_2^{\prime}\leq \frac{Q}{\beta d_2}\\(q_1^{\prime},\beta d_1)=1, (q_2^{\prime},\beta d_2)=1\\(q_1^{\prime}, q_2^{\prime})=1}}\mu(q_1^{\prime})^2\mu(q_2^{\prime})^2H\left(\frac{3r\delta(p)Q^2}{\beta\pi^2q_1^{\prime}q_2^{\prime}} \right)=&\frac{6P_{\beta,d_1,d_2}}{\pi^2}\int_{0}^{\frac{Q}{\beta d_1}}\int_{0}^{\frac{Q}{\beta d_2}}H\left(\frac{3r\delta(p)Q^2}{\beta\pi^2xy}\right)dxdy\\&+\BigO{(\tau(\beta d_1)+\tau(\beta d_2))Q^{\frac{3}{2}}\log^2Q},\numberthis\label{sum16}
\end{align*}
where
\[P_{\beta,d_1,d_2}=\frac{\phi(\beta d_1)\phi(\beta d_2)}{\beta^2d_1d_2}\prod_{p|\gcd(\beta d_1, \beta d_2)}\left(1-\frac{1}{p^2}\right)^{-1}\prod_{\substack{p \\(p,\beta d_1d_2)=1}}\left(1-\frac{1}{p(p+1)}\right)\left(1-\frac{1}{p^2+p-1}\right).\]
The main term in \eqref{sum16}, by a suitable change of variable, can be expressed as 
\[\frac{6P_{\beta,d_1,d_2}Q^2}{\pi^2}\int_{0}^{\frac{1}{\beta d_1}}\int_{0}^{\frac{1}{\beta d_2}}H\left(\frac{3r\delta(p)}{\beta\pi^2xy}\right)dxdy.\]
Returning to the sum in \eqref{sum5}, we get
\[S=\frac{6Q^2}{\pi^2}\sum_{\substack{d_1, d_2,\beta, r\geq 1\\ \beta d_1d_2r<\mathcal{C}_{\Lambda}\\(\beta, d_1d_2)=1}}\beta\mu(d_1)\mu(d_2)\mu(\beta)^2P_{\beta,d_1,d_2}\int_{0}^{\frac{1}{\beta d_1}}\int_{0}^{\frac{1}{\beta d_2}}H\left(\frac{3r\delta(p)}{\beta\pi^2xy}\right)dxdy+\BigOC{Q^{\frac{3}{2}}\log^2Q}. \numberthis\label{sum17}\]
Since Supp $H\subset (0,\Lambda)$, we put $\lambda=\frac{3r\delta(p)}{\beta\pi^2xy}$ then the double integral in the above sum becomes
\begin{align*}
\mathfrak{I}_H:&=\int_{0}^{\frac{1}{\beta d_1}}\int_{0}^{\frac{1}{\beta d_2}}H\left(\frac{3r\delta(p)}{\beta\pi^2xy}\right)dxdy\\
&=\frac{3r\delta(p)}{\beta\pi^2}\int_{0}^{\frac{1}{\beta d_1}}\int_{\frac{3rd_2\delta(p)}{\pi^2 x}}^{\Lambda}\frac{H(\lambda)}{\lambda^2x}d\lambda dx\\
&=\frac{3r\delta(p)}{\beta\pi^2}\int_{\frac{3r\beta d_1d_2\delta(p)}{\pi^2 }}^{\Lambda}\int_{\frac{3rd_2\delta(p)}{\pi^2 \lambda}}^{\frac{1}{\beta d_1}}\frac{H(\lambda)}{\lambda^2x}dx d\lambda \\
    &=\frac{3r\delta(p)}{\beta\pi^2}\int_{\frac{3r\beta d_1d_2\delta(p)}{\pi^2 }}^{\Lambda}\frac{H(\lambda)}{\lambda^2}\log\left({\frac{\pi^2\lambda}{3r\beta d_1d_2\delta(p)}}\right)d\lambda.
\end{align*}
Inserting $\mathfrak{I}_H$ in \eqref{sum17}, we have
\small\begin{align*}
    S=&\frac{18Q^2\delta(p)}{\pi^4}\sum_{\substack{d_1, d_2,\beta, r\geq 1\\ \beta d_1d_2r<\mathcal{C}_{\Lambda}\\(\beta, d_1d_2)=1}}r\mu(d_1)\mu(d_2)\mu(\beta)^2P_{\beta,d_1,d_2}\int_{\frac{3r\beta d_1d_2\delta(p)}{\pi^2 }}^{\Lambda}\frac{H(\lambda)}{\lambda^2}\log\left({\frac{\pi^2\lambda}{3r\beta d_1d_2\delta(p)}}\right)d\lambda\\&+\BigOC{Q^{\frac{3}{2}}\log^2Q}\\
    =&\frac{18Q^2\delta(p)}{\pi^4}\sum_{1\leq m<\mathcal{C}_{\Lambda}}\int_{\frac{3m\delta(p)}{\pi^2 }}^{\Lambda}\frac{H(\lambda)}{\lambda^2}\log\left({\frac{\pi^2\lambda}{3m\delta(p)}}\right)d\lambda\sum_{\substack{\beta d_1d_2r=m\\(\beta, d_1d_2)=1}}r\mu(d_1)\mu(d_2)\mu(\beta)^2P_{\beta,d_1,d_2}\\&+\BigOC{Q^{\frac{3}{2}}\log^2Q}.\numberthis\label{sum18}\\
\end{align*}
Now, using the fact that $(\beta,d_1d_2)=1$, product term $P_{\beta,d_1,d_2}$ in the inner sum in \eqref{sum18} can be expressed as
\[P_{\beta,d_1,d_2}=P_{\beta}P_{d_1,d_2}, \]
where
\[P_{\beta}=\prod_{p|\beta}\left(\frac{p(p-1)}{p^2+p-2}\right), \]
and
\[P_{d_1,d_2}=\frac{\phi(d_1)\phi(d_2)}{d_1d_2}\prod_{p|\gcd(d_1,d_2)}\left(1-\frac{1}{p^2}\right)^{-1}\prod_{\substack{p\\(p,d_1d_2)=1}}\left(1-\frac{1}{p(p+1)}\right)\left(1-\frac{1}{p^2+p-1}\right),\]
which can be further expressed as 
\[P_{d_1,d_2}=P_{d_1,d_2}^{1}P_{d_1,d_2}^2,\]
where
\[P_{d_1,d_2}^1=\prod_{p|d_1}\left(1-\frac{1}{p}\right)\prod_{\substack{p\\(p,d_1)=1}}\left(1-\frac{1}{p(p+1)}\right)\left(1-\frac{1}{p^2+p-1}\right),\]
and
\[P_{d_1,d_2}^2=\prod_{p|d_2}\left(1-\frac{1}{p}\right)\prod_{p|\gcd(d_1,d_2)}\left(1-\frac{1}{p^2}\right)^{-1}\prod_{\substack{p|d_2\\(p,d_1)=1}}\left(1-\frac{1}{p(p+1)}\right)\left(1-\frac{1}{p^2+p-1}\right).\]
To estimate the inner sum in the main term of \eqref{sum18}, we write
\begin{align*}
    F(m):&=\sum_{\substack{\beta d_1d_2r=m\\(\beta, d_1d_2)=1}}r\mu(d_1)\mu(d_2)\mu(\beta)^2P_{\beta}P_{d_1,d_2}\\
    &=m\sum_{d_1|m}\frac{\mu(d_1)}{d_1}\sum_{d_2|\frac{m}{d_1}}\frac{\mu(d_2)}{d_2}P_{d_1,d_2}\sum_{\substack{\beta|\frac{m}{d_1d_2}\\(\beta, d_1d_2)=1}}\frac{\mu(\beta)^2}{\beta}P_{\beta}.\numberthis\label{FP(m)}
\end{align*}
We denote the innermost sum in \eqref{FP(m)} by $F_{(m,d_1,d_2)}$ and one can observe that it is multiplicative, so we evaluate it on the prime powers 
\begin{align*}
    F_{(m,d_1,d_2)}:&=\sum_{\substack{\beta|\frac{m}{d_1d_2}\\(\beta, d_1d_2)=1}}\frac{\mu(\beta)^2}{\beta}P_{\beta}
    =\sum_{\substack{\beta|\frac{m}{d_1d_2}\\(\beta, d_1d_2)=1}}\frac{\mu(\beta)^2}{\beta}\prod_{p|\beta}\left(\frac{p(p-1)}{p^2+p-2}\right)\\
    &=\prod_{\substack{p|\frac{m}{d_1d_2}\\(p,d_1d_2)=1}}\left(1+\frac{p-1}{p^2+p-2}\right).\numberthis\label{product}
\end{align*}
So, \eqref{product} in conjunction with \eqref{FP(m)} gives
\begin{align*}
    F(m)&=m\sum_{d_1|m}\frac{\mu(d_1)}{d_1}\sum_{d_2|\frac{m}{d_1}}\frac{\mu(d_2)}{d_2}P_{d_1,d_2}\prod_{\substack{p|\frac{m}{d_1d_2}\\(p,d_1d_2)=1}}\left(1+\frac{p-1}{p^2+p-2}\right)\\
    &=m\sum_{d_1|m}\frac{\mu(d_1)}{d_1}P_{d_1,d_2}^1\sum_{d_2|\frac{m}{d_1}}\frac{\mu(d_2)}{d_2}\prod_{\substack{p|\frac{m}{d_1d_2}\\(p,d_1d_2)=1}}\left(1+\frac{p-1}{p^2+p-2}\right)P_{d_1,d_2}^2.\numberthis\label{F1(m)}
\end{align*}
Similarly, the inner sum in \eqref{F1(m)} is determined at prime powers, we write
\begin{align*}
    F_{(m,d_1)}:=&\sum_{d_2|\frac{m}{d_1}}\frac{\mu(d_2)}{d_2}\prod_{\substack{p|\frac{m}{d_1d_2}\\(p,d_1d_2)=1}}\left(1+\frac{p-1}{p^2+p-2}\right)P_{d_1,d_2}^2\\
    =&\prod_{\substack{p|\frac{m}{d_1}\\(p,d_1)=1}}\left(1-\frac{(p-1)(p^2+p-2)}{p^3(p+1)}\prod_{\substack{p^{\prime}|\frac{m}{d_1p}\\(p^{\prime},d_1p)=1}}\left(1+\frac{p^{\prime}-1}{p{^{\prime}}^2+p^{\prime}-2}\right)\right)\\&\times\prod_{\substack{p|\frac{m}{d_1}\\p|d_1}}\left(1-\frac{1}{p+1}\prod_{\substack{p^{\prime}|\frac{m}{d_1p}\\(p^{\prime},d_1p)=1}}\left(1+\frac{p^{\prime}-1}{p{^{\prime}}^2+p^{\prime}-2}\right)\right).\numberthis\label{F2(m)}
\end{align*}
So, \eqref{F2(m)} together with \eqref{F1(m)} gives the required $F(m)$ as defined in \eqref{F(m)}. Inserting $F(m)$ in \eqref{sum18}, we obtain
\begin{align*}
    S&=\frac{18Q^2\delta(p)}{\pi^4}\sum_{1\leq m<\mathcal{C}_{\Lambda}}F(m)\int_{\frac{3m\delta(p)}{\pi^2 }}^{\Lambda}\frac{H(\lambda)}{\lambda^2}\log\left({\frac{\pi^2\lambda}{3m\delta(p)}}\right)d\lambda+\BigOC{Q^{\frac{3}{2}}\log^2Q}\\
    &=\frac{18Q^2\delta(p)}{\pi^4}\int_{0}^{\Lambda}\frac{H(\lambda)}{\lambda^2}\sum_{1\leq m<\mathcal{C}_{\Lambda}}F(m)\max\left(0,\log\left({\frac{\pi^2\lambda}{3m\delta(p)}}\right)\right)d\lambda+\BigOC{Q^{\frac{3}{2}}\log^2Q}\\
    &=\frac{3Q^2\delta(p)}{\pi^2}\int_{0}^{\Lambda}H(\lambda)\mathfrak{g}_2(\lambda)d\lambda+\BigOC{Q^{\frac{3}{2}}\log^2Q},
\end{align*}
where the function $\mathfrak{g}_2(\lambda)$ is defined in Theorem \ref{main result}.
\[\frac{S}{\mathscr{N}_{Q}}=\int_{0}^{\Lambda}H(\lambda)\mathfrak{g}_2(\lambda)d\lambda+\BigOC{\frac{\log^2Q}{Q^\frac{1}{2}}}.\]
Now we approximate $H$ by the characteristic function of $(0,\Lambda)$, using the standard approximation argument, we get
\[\lim_{Q\to\infty}S(\Lambda)=\int_{0}^{\Lambda}\mathfrak{g}_2(\lambda)d\lambda. \]
Hence the limiting pair correlation function of $\mathcal{F}_{Q,2}$ is $\mathfrak{g}_2(\lambda)$.

\section{Acknowledgement}
The first author acknowledges the support from the University Grants Commission, Department of Higher Education, Government of India, under NTA Ref. no. 191620135578. The research of the second author was supported by the Science and Engineering Research Board, Department of Science and Technology, Government of India, under grant SB/S2/RJN-053/2018. We are grateful to the referee for many valuable suggestions in an earlier version of the paper.

\bibliographystyle{amsalpha} 
\bibliography{reference}
\end{document}